\documentclass[a4paper]{llncs}

\usepackage{amsmath,amssymb}
\usepackage{etex}
\usepackage[latin1]{inputenc}
\usepackage[OT1]{fontenc}
\usepackage[english]{babel}
\usepackage{graphicx,subfigure,float}
\usepackage{nicefrac}
\usepackage{mathrsfs}
\usepackage[linesnumbered,ruled,vlined]{algorithm2e}
\usepackage{algpseudocode}
\usepackage{pstricks}
\usepackage{pst-func,pst-eucl,pstricks-add}

\DeclareMathOperator{\rank}{rank}
\DeclareMathOperator{\cv}{conv}

\DeclareMathOperator*{\argmax}{arg\,max}

\newcommand{\Rbb}{\mathbb{R}}
\newcommand{\Rn}{\mathbb{R}^n}

\newcommand{\Zbb}{\mathbb{Z}}

\newcommand{\pin}[2]{\langle#1,#2\rangle}
\newcommand{\st}{\text{s.t.\ }}

\def\keywords#1{\noindent{\bf Keywords:}\enspace\ignorespaces{#1}}

\begin{document}

\mainmatter

\title{A coordinate ascent method for solving semidefinite relaxations of non-convex quadratic integer programs}
\titlerunning{A coordinate ascent method for solving SDP relaxations of non-convex QIP}
\author{Christoph Buchheim\inst{1}
  \and
Maribel Montenegro\inst{1}
  \and
Angelika Wiegele\inst{2}}
\institute{Fakult\"at f\"ur Mathematik, Technische Universit\"at Dortmund,
  Germany, \email{christoph.buchheim@tu-dortmund.de, maribel.montenegro@math.tu-dortmund.de}
\and Department of Mathematics, Alpen-Adria-Universit\"at Klagenfurt, Austria \email{angelika.wiegele@aau.at}}
\date{}
\maketitle

\begin{abstract}
  We present a coordinate ascent method for a class of semidefinite
  programming problems that arise in non-convex quadratic integer
  optimization. These semidefinite programs are characterized by a
  small total number of active constraints and by low-rank constraint
  matrices. We exploit this special structure by solving the dual
  problem, using a barrier method in combination with a
  coordinate-wise exact line search. The main ingredient of our
  algorithm is the computationally cheap update at each iteration and
  an easy computation of the exact step size. Compared to interior
  point methods, our approach is much faster in obtaining strong dual
  bounds. Moreover, no explicit separation and reoptimization is
  necessary even if the set of primal constraints is large, since in
  our dual approach this is covered by implicitly considering all
  primal constraints when selecting the next coordinate.

  \medskip


\end{abstract}

\section{Introduction}

The importance of Mixed-Integer Quadratic Programming (MIQP) lies in
both theory and practice of mathematical optimization. On one hand, a
wide range of problems arising in practical applications can be
formulated as MIQP. On the other hand, it is the most natural
generalization of Mixed-Integer Linear Programming (MILP).  However,
it is well known that MIQP is NP-hard, as it contains
MILP as a special case. Moreover, contrarily to what happens in MILP,
the hardness of MIQP is not resolved by relaxing the integrality
requirement on the variables: while convex quadratic problems can be
solved in polynomial time by either the ellipsoid method
\cite{KozlovTarasov(1980)} or interior point methods
\cite{KapoorVaidya(1986),YeTse(1989)}, the general problem of
minimizing a non-convex quadratic function over a box is NP-hard, even
if only one eigenvalue of the Hessian is negative~\cite{pardalosvavasis:91}.

Buchheim and Wiegele \cite{BuchheimWiegele(2013)} proposed the use of
semidefinite relaxations and a specialized branching scheme (Q-MIST)
for solving unconstrained non-convex quadratic minimization problems
where the variable domains are arbitrary closed subsets of
$\Rbb$. Their work is a generalization of the well-known semidefinite
programming approach to the maximum cut problem or, equivalently, to
unconstrained quadratic minimization over variables in the domain
$\{-1,1\}$.
Q-MIST needs to solve a semidefinite program (SDP) at
each node of the branch-and-bound tree, which can be done using any
standard SDP solver. In~\cite{BuchheimWiegele(2013)}, an interior
point method was used for this task, namely the CSDP
library~\cite{csdp}. It is well-known that interior point algorithms
are theoretically efficient to solve SDPs, they are able to solve
small to medium size problems with high accuracy, but they are memory
and time consuming for large scale instances.

A related approach to solve the same kind of non-convex quadratic
problems was presented by Dong~\cite{HongboDong(2014)}. A convex
quadratic relaxation is produced by means of a cutting surface
procedure, based on multiple diagonal perturbations. The separation
problem is formulated as a semidefinite problem and is solved by
coordinate-wise optimization methods. More precisely, the author
defines a barrier problem and solves it using coordinate descent
methods with exact line search. Due to the particular structure of the
problem, the descent direction and the step length can be computed by closed formulae, and fast updates are
possible using the Sherman-Morrison formula. Computational results
show that this approach produces lower bounds as strong as the ones
provided by \mbox{Q-MIST} and it runs much faster for instances of large
size.

In this paper, we adapt and generalize the coordinate-wise
approach of~\cite{HongboDong(2014)} in order to solve the
dual of the SDP relaxation arising in the Q-MIST approach. In our
setting, it is still true that an exact coordinate-wise line search
can be performed efficiently by using a closed-form expression, based
on the Sherman-Morrison formula. Essentially, each iteration of the
algorithm involves the update of one coordinate of the vector of dual
variables and the computation of an inverse of a matrix that changes
by a rank-two constraint matrix when changing the value of the
dual variable.
Altogether, our
approach fully exploits the specific structure of our problem, namely a small total
number of (active) constraints and low-rank constraint matrices of
the semidefinite relaxation. Furthermore, in our model the set of dual variables
can be very large, so that the selection of the best
coordinate requires more care than
in~\cite{HongboDong(2014)}. However, our new approach is much more
efficient than the corresponding separation approach for the primal
problem described in~\cite{BuchheimWiegele(2013)}.

\section{Preliminaries}
\label{Sec:qmist}

We consider non-convex quadratic mixed-integer optimization problems of the form
\begin{align}
\label{P:Quadratic}
\min &\quad x^\top \hat Q x +\hat l^\top x+ \hat c\nonumber\\
\st &\quad x\in D_1\times\dots\times D_n\;,
\end{align}
where $\hat Q\in \Rbb^{n\times n}$ is symmetric but not necessarily positive
semidefinite,~$\hat l\in
\Rn$, $\hat c\in \Rbb$, and $D_i=\{l_i,\dots,u_i\} \subseteq\Zbb$ is finite for
all~$i=1,\dots,n$. Buchheim and Wiegele
\cite{BuchheimWiegele(2013)} have studied the more general case where
each~$D_{i}$ is an arbitrary closed subset of~$\Rbb$. The authors have implemented a branch-and-bound approach
called Q-MIST, it mainly consists in reformulating
Problem~\eqref{P:Quadratic} as a semidefinite optimization problem and solving a
relaxation of the transformed problem within a branch-and-bound
framework.
In this section, first we describe how to obtain a semidefinite relaxation of 
Problem~\eqref{P:Quadratic}, then we formulate it in a matrix form and compute 
the dual problem.

\subsection{Semidefinite relaxation}
Semidefinite relaxations for quadratic optimization problems can already be found
in an early paper of Lov\'asz in~1979~\cite{Lovasz(1979)}, but it was
not until the work of Goemans and Williamson
in~1995~\cite{GoemansWilliamson(1995)} that they started to catch
interest.  The basic idea is as follows: given any vector~$x\in\Rn$,
the matrix~$xx^\top\in \Rbb^{n\times n}$ is rank-one, symmetric and
positive semidefinite. In particular, also the augmented matrix
\[
\ell(x):=\begin{pmatrix}
1\\
x
\end{pmatrix}
\begin{pmatrix}
1\\
x
\end{pmatrix}^\top
=
\begin{pmatrix}
1~ & x^\top\\
x~ & xx^\top
\end{pmatrix}
\in\Rbb^{(n+1)\times(n+1)}
\]
is positive semidefinite. This well-known fact leads to semidefinite
reformulations of various quadratic problems. Defining a matrix
\[
Q
:=
\begin{pmatrix}
\hat c & \tfrac{1}{2}\hat l^\top\\
\tfrac{1}{2}\hat l & \hat Q
\end{pmatrix},
\]
Problem~\eqref{P:Quadratic} can be rewritten as
\begin{align*}
\min & \quad \pin{Q}{X} \nonumber\\
\st & \quad X\in \ell(D_1\times\dots\times D_n)\;,
\end{align*}
so that it remains to investigate the set~$\ell(D_1\times\dots\times
D_n)$. The following result was proven in~\cite{BuchheimWiegele(2013)}.
\begin{theorem}
\label{T:setl(x)}
Let $X\in \Rbb^{(n+1)\times(n+1)}$ be symmetric. Then $X\in \ell(D_1\times\dots\times D_n)$ if and only if
\begin{itemize}
\item[(a)] $(x_{i0},x_{ii})\in P(D_i):=\cv\{(u,u^2)\mid u\in D_i\}$ for all $i=1,\dots,n$,  
\item[(b)] $x_{00}=1$,
\item[(c)] $\rank(X)=1$, and
\item[(d)] $X\succeq 0$.
\end{itemize}
\end{theorem}

\noindent
We derive that the following optimization problem is a
convex relaxation of~\eqref{P:Quadratic}, obtained by dropping
the rank-one constraint of Theorem~\ref{T:setl(x)}\,(c):
\begin{align}
\label{P:SDP-relax}
\min \quad \pin{Q}{X}& \nonumber\\ \st \, (x_{i0},x_{ii})&\in P(D_i)
\quad \forall i=1,\dots n\\ x_{00}&=1 \nonumber\\ X&\succeq0 \nonumber
\end{align}
This is an SDP, since the
constraints~$(x_{i0},x_{ii})\in P(D_i)$ can be replaced by a set of
linear constraints, as discussed in the next section.

\subsection{Matrix formulation}

In the case of finite~$D_i$ considered here, the set~$P(D_{i})$ is a polytope in~$\Rbb^2$ with~$|D_i|$ many extreme
points. It can thus be described equivalently by a set of~$|D_i|$ linear
inequalities.
\begin{lemma}
  For~$D_i=\{l_i,\dots,u_i\}$, the polytope~$P(D_i)$ is completely described by lower
  bounding facets $-x_{ii}+(j+(j+1))x_{0i}\leq j(j+1)$ for
  $j=l_{i},l_{i}+1,\dots,u_{i}-1$ and one upper bounding facet
  $x_{ii}-(l_{i}+u_{i})x_{0i}\leq -l_iu_i$.
\end{lemma}
Exploiting~$x_{00}=1$, we may rewrite the polyhedral description
of~$P(D_i)$ presented in the previous lemma as
\begin{align*}
  (1-j(j+1))x_{00}-x_{ii}+(j+(j+1))x_{0i} & \leq 1, ~ j=l_{i},l_{i}+1,\dots,u_{i}-1\\
  (1+l_iu_i)x_{00}+x_{ii}-(l_{i}+u_{i})x_{0i} & \leq 1\;.
\end{align*}
We write the resulting inequalities in matrix form as
$ \pin{A_{ij}}{X}\leq 1$.
To keep analogy with the facets, the index~$ij$ represents the inequalities
corresponding to lower bounding facets if $j=l_{i},l_{i}+1,\dots,u_{i}-1$
whereas~$j= u_{i}$ corresponds to the upper facet; see
Figure~\ref{Fig:inequality}.

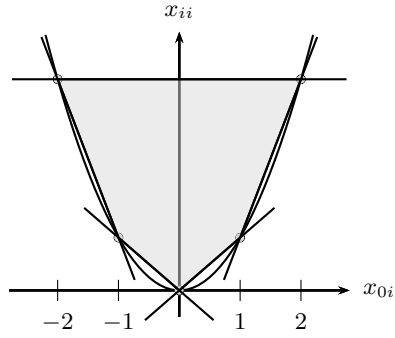
\begin{figure}
\centering
\psset{xunit=0.8\psunit,yunit=0.7\psunit}
\begin{pspicture}[showgrid=false](-2.8,-0.5)(2.8,5)
\psaxes[ticks=x,labels=x]{->}(0,0)(-2.8,-0.5)(2.8,4.9)[$x_{0i}$,0][$x_{ii}$,90]
  \psplot{-2.2}{2.2}{x 2 exp}
  \pstGeonode[PointSymbol=o,PointName=none,CurveType=polygon,fillstyle=solid,fillcolor=black!15,opacity=0.5](0,0){A}(1,1){B}(2,4){E}(-2,4){D}(-1,1){C}
\pstLineAB[nodesep=-0.6]{A}{B}
\pstLineAB[nodesep=-0.6]{A}{C}
\pstLineAB[nodesep=-0.6]{D}{E}
\pstLineAB[nodesep=-0.6]{D}{C}
\pstLineAB[nodesep=-0.6]{B}{E}	  
\end{pspicture}
\caption{\label{Fig:inequality}\small The polytope
  $P(\{-2,-1,0,1,2\})$. Lower bounding facets are indexed, from left to right,
  by $j= -2,-1,0,1$, the upper bounding facet is indexed by $2$.}
\end{figure}

Moreover, we write the constraint~$x_{00}=1$ in 
matrix form as $\pin{A_0}{X}=1$, where~$A_0:=e_0e_0^\top$. 
In summary, Problem~\eqref{P:SDP-relax} can now be stated as
\begin{align}
\label{P:SDP-matrix}
\min~~\quad \pin{Q}{X} \nonumber \\
\st \quad \pin{A_{0}}{X}&=1 \\ 
\pin{A_{ij}}{X}&\leq 1 \quad\forall j=l_i,\dots,u_i, \forall i=1,\dots,n \nonumber \\
X&\succeq0. \nonumber
\end{align}
The following simple observation is crucial for our algorithm
presented in the following section.
\begin{lemma}
  \label{lemma:rank}
  The constraint matrix~$A_0$ has rank one.
  All constraint matrices~$A_{ij}$ have rank one or two. The rank
  of~$A_{ij}$ is one if and only if~$j=u_i$ and $u_i-l_i=2$.
\end{lemma}

\subsection{Dual problem}

In order to derive the dual problem of~\eqref{P:SDP-matrix}, we define
\[
\mathcal{A}(X):=
\begin{pmatrix}
\pin{A_{0}}{X}\\ 
\pin{A_{ij}}{X}_{j\in \{l_i,\dots,u_i\}, i\in \{1, \dots, n\}}
\end{pmatrix}
\]
and associate a dual variable $y_{0}\in \Rbb$ with the
constraint $\pin{A_{0}}{X}=0$ as well as dual variables $y_{ij}\leq 0$ with
$\pin{A_{ij}}{X}\leq 1$, for~$j\in
\{l_{i},\dots,u_{i}\}$ and~$i\in \{1,\dots,n\}$. We then define $y\in
\Rbb^{m+1}$ as
\[
y:=
\begin{pmatrix}
y_0\\
(y_{ij})_{j\in \{l_i,\dots,u_i\}, i\in \{1, \dots, n\}}
\end{pmatrix}.
\]
The dual semidefinite program of Problem~\eqref{P:SDP-matrix} is
\begin{align}
\label{P:Dsdp}
\max~~~\quad \pin{b}{y}\quad & \nonumber \\
\st \quad  Q - \mathcal{A}^\top y &\succeq 0 \\
y_{0} &\in \Rbb \nonumber\\
y_{ij}&\leq0 \quad \forall j=l_{i},\dots,u_i, \forall i=1,\dots, n\nonumber,
\end{align}
the vector $b\in \Rbb^{m+1}$ being the all-ones vector. It is easy to
verify that the primal problem~\eqref{P:SDP-matrix} is strictly feasible
if~$|D_i|\ge 2$ for all~$i=1,\dots,n$, so that strong
duality holds in all non-trivial cases.

We conclude this section by emphasizing some characteristics of any
feasible solution of Problem~\eqref{P:SDP-matrix}.
\begin{lemma}
\label{Lem:ActiveIneq}
Let $X^*$ be a feasible solution of
Problem~\eqref{P:SDP-matrix}. For~$i\in\{1,\dots,n\}$, consider the
active set
\[
\mathscr{A}_i=\{j\in\{l_i,\dots,u_i\}\mid
\pin{A_{ij}}{X^*}=1\}
\]
corresponding to variable~$i$. Then
\begin{itemize}
\item[(i)]  for all~$i\in\{1,\dots,n\}$, $|\mathscr{A}_i|\le 2$, and
\item[(ii)] if $|\mathscr{A}_i|=2$, then
  $x_{ii}^*={x_{0i}^*} ^2$ and $x_{i0}^*\in D_i$.
\end{itemize}
\end{lemma}

\begin{proof}
The polytope $P(D_i)$ is two-dimensional with non-degenerate vertices.
Due to the way the inequalities $\pin{A_{ij}}{X}\leq 1$ are defined it
is impossible to have more than two inequalities intersecting at one
point. Therefore, a given
point~$(x_{ii},x_{i0})\in P(D_i)$ satisfies zero, one, or two
inequalities with equality. In the last case, we
have~$x_{ii}=x_{i0}^2$ by construction, which implies~$x_{i0}\in D_i$.
\qed
\end{proof}

\noindent
For the dual problem~\eqref{P:Dsdp}, Lemma~\ref{Lem:ActiveIneq}\,(i)
means that at most~$2n+1$ out of the~$m+1$ variables can be non-zero
in an optimal solution. Clearly, such a small number of non-zero
variables is beneficial in a coordinate-wise optimization
method. Moreover, by Lemma~\ref{Lem:ActiveIneq}\,(ii), if two dual
variables corresponding to the same primal variable are non-zero in an
optimal dual solution, then this primal variable will obtain an
integer feasible value in the optimal primal solution.

\section{A coordinate ascent method}
\label{Sec:cd}

We aim at solving the dual problem~\eqref{P:Dsdp} by coordinate-wise
optimization, in order to obtain fast lower bounds to be used inside
the branch-and-bound framework Q-MIST. Our approach is motivated by an
algorithm proposed by Dong~\cite{HongboDong(2014)}. The author
formulates Problem~\eqref{P:Quadratic} as a convex quadratically
constrained problem, and devises a cutting surface procedure based on
diagonal perturbations to construct convex relaxations. The separation
problem turns out to be a semidefinite problem with convex non-smooth
objective function, and it is solved by a primal barrier coordinate
minimization algorithm with exact line search.

The dual Problem~\eqref{P:Dsdp} has a similar structure to the
semidefinite problem solved in~\cite{HongboDong(2014)}, therefore
similar ideas can be applied. Our SDP is more general however, it
contains more general constraints with matrices of rank two (instead
of one) and most of our variables are constrained to be
non-positive. Another difference is that we deal with a very large
number of constraints, out of which only a few are non-zero
however. On the other hand, our objective function is linear, which is
not true for the problem considered in~\cite{HongboDong(2014)}.

As a first step, we introduce a penalty term modelling the
semidefinite constraint~$Q-\mathcal{A}^\top y\succeq 0$ of Problem~\eqref{P:Dsdp}
and obtain
\begin{align}
\label{P:Dsdp-barrier}
\max~~~ \quad f(y;\sigma)&:=\pin{b}{y}+\sigma \log\det(Q-\mathcal{A}^\top y)  \nonumber \\
\st \quad Q-\mathcal{A}^\top y &\succ 0 \\
 y_{0} &\in \Rbb  \nonumber \\
y_{ij}&\leq0 \quad \forall j=l_{i},\dots,u_i, \forall i=1,\dots, n\nonumber
\end{align}
for $\sigma>0$.
The gradient of the objective function of Problem~\eqref{P:Dsdp-barrier} is
\[
\nabla_{y}f(y;\sigma)=b-\sigma\mathcal{A}((Q-\mathcal{A}^\top y)^{-1}).
\]
For the following, we denote $W :=(Q-\mathcal{A}^\top y)^{-1}$, so that
\begin{equation}
\label{gradient}
\nabla_{y}f(y;\sigma)=b-\sigma\mathcal{A}(W)\;.
\end{equation}
We will see later that, using the Sherman-Morrison formula, the
matrix~$W$ can be updated quickly when changing the value of a dual
variable, which is crucial for the performance of the algorithm
proposed. We begin by describing a general algorithm to
solve~\eqref{P:Dsdp-barrier} in a coordinate maximization manner.
In the following, we explain each step of this algorithm in detail.

\begin{algorithm}[H]
\NoCaptionOfAlgo
\caption{\textbf{Outline of a barrier coordinate ascent algorithm for Problem~\eqref{P:Dsdp}}}

\label{Algo:general}
 \textbf{Starting point:} choose any feasible solution~$y$ of~\eqref{P:Dsdp-barrier}\;
 \textbf{Direction:} choose a coordinate direction $e_{ij}$\;
 \textbf{Step size:} using exact line search, determine the step length~$s$\;
 \textbf{Move along chosen coordinate:} $y \gets y+se_{ij}$\;
 \textbf{Update} the matrix $W$ accordingly\;
 \textbf{Decrease} the penalty parameter $\sigma$\;
 \textbf{Go to (2)}, unless some stopping criterion is satisfied\;
\end{algorithm}

\subsection{Definition of a starting point}

If $Q\succ0$, we can safely choose $y^{(0)}=0$ as starting
point. Otherwise, define~$a\in\Rbb^n$ by~$a_i=(A_{iu_i})_{0i}$ for
$i=1,\dots,n$. Moreover, define
\begin{align*}
\tilde{y}&:=\min\{\lambda_{min}(\hat Q)-1,0\}, \\
y_0&:=\hat c-\tilde y \sum_{i=1}^n(1+l_iu_i)-1-(\tfrac{1}{2}\hat l-\tilde{y}a)^{\top}(\tfrac{1}{2}\hat l-\tilde{y}a),
\end{align*}
and $y^{(0)}\in\Rbb^{m+1}$ as
\[
y^{(0)}:=
\begin{pmatrix}
y_{0}\\
(y_{ij})_{j\in \{l_i,\dots,u_i\}, i\in \{1, \dots, n\}}
\end{pmatrix},
\quad
y_{ij}= 
\begin{cases}
\tilde{y}, &  j=u_i, i=1,\dots, n\\
0, & \text{otherwise.}
\end{cases}
\]
Then the following lemma holds; the proof can be found in Appendix~\ref{app:feas}.
\begin{lemma}
  \label{lem:initial}
  The vector~$y^{(0)}$ is feasible for \eqref{P:Dsdp-barrier}.
\end{lemma}

\subsection{Choice of an ascent direction}
\label{sec:direction}

We improve the objective function coordinate-wise: at each
iteration~$k$ of the algorithm, we choose an ascent
direction~$e_{ij^{(k)}}\in\Rbb^{m}$ where~$ij^{(k)}$ is the coordinate
of the gradient with maximum absolute value
\begin{equation}
\label{Eqn:choosecoord}
ij^{(k)}:=\argmax_{ij}|\nabla_{y}f(y;\sigma)_{ij}|\;.
\end{equation}
However, moving a coordinate~$ij$ to a positive direction is allowed
only if~$y_{ij}<0$, so that the coordinate~$ij^{(k)}$
in~\eqref{Eqn:choosecoord} has to be chosen among those satisfying
\begin{equation*}
(\nabla_{y}f(y;\sigma)_{ij}>0 \text{ and } y_{ij}<0) \quad
\text{ or }
\quad\nabla_{y}f(y;\sigma)_{ij}<0\;.
\end{equation*}
The entries of the gradient depend on the type of
inequality. By~\eqref{gradient}, we have
\begin{align*}
\nabla_{y}f(y;\sigma)_{ij}&=1-\sigma \pin{W}{A_{ij}}.
\end{align*}
The number of lower bounding facets for a single primal variable~$i$
is~$u_i-l_i$, which is not polynomial in the input size from a
theoretical point of view. From a practical point of view, a large
domain~$D_i$ may slow down the coordinate selection if all potential
coordinates have to be evaluated explicitly.

However, the regular structure of the gradient entries corresponding
to lower bounding facets for variable~$i$ allows to limit the search
to at most two candidates per variable. To this end, we define the
function
$$\varphi_i(j):=1-\sigma \pin{W}{A_{ij}}=1-\sigma\big((1-j(j+1))W_{00}+(2j+1)W_{i0} -W_{ii}\big)$$
and aim at finding a minimizer of~$|\varphi|$ over~$\{l_i,\dots,u_i-1\}$.
As $\varphi_i$ is a univariate quadratic function, we can restrict our
search to at most three candidates, namely the bounds~$l_i$
and~$u_i-1$ and the rounded global minimizer of~$\varphi_i$, if it
belongs to~$l_i,\dots,u_i-1$; the latter is
\[
\left\lceil\tfrac{W_{i0}}{W_{00}}-\tfrac{1}{2}\right\rfloor\;.
\]
In summary, taking into account also the upper bounding facets and the
coordinate zero, we need to test at most~$4n+1$ candidates in order to
solve~\eqref{Eqn:choosecoord}, independently of the bounds~$l_i$ and~$u_i$.

\subsection{Computation of the step size} 
\label{sec:step}

We compute the step size~$s^{(k)}$ by exact line search in the chosen 
direction. For this, we need to solve the following one-dimensional 
maximization problem
\begin{equation}
\label{P:Linesearch}
s^{(k)}= \argmax_{s}\{f(y^{(k)} +se_{ij^{(k)}};\sigma)\mid 
Q-\mathcal{A}^\top (y^{(k)}+se_{ij^{(k)}}) \succ 0 , s\leq -y_{ij^{(k)}} \}
\end{equation}
unless the chosen coordinate is zero, in which case the upper bound
on~$s$ is dropped. 
Note that~$s\mapsto f(y^{(k)} +se_{ij^{(k)}};\sigma)$ is
strictly concave on
$$\{s\in\Rbb\mid Q-\mathcal{A}^\top (y^{(k)}+se_{ij^{(k)}}) \succ
0\}\;.$$
 By the first order optimality conditions, we thus need to find
the unique~$s^{(k)}\in\Rbb$ satisfying the semidefinite constraint
$Q-\mathcal{A}^\top (y^{(k)}+s^{(k)}e_{ij^{(k)}}) \succ 0$ such that
either
\[
\nabla_{s}f(y^{(k)} +s^{(k)}e_{ij^{(k)}};\sigma)=0 
\quad\text {and}\quad y_{{ij}^{(k)}}+s^{(k)}\leq 0
\]
or 
\[
\nabla_{s}f(y^{(k)} +s^{(k)}e_{ij^{(k)}};\sigma)>0
\quad\text {and}\quad s^{(k)}=- y_{ij^{(k)}}^{(k)}.
\]
In order to simplify the notation, we omit the superindex~$(k)$ in the
following. From the definition,
\begin{align*}
f(y +se_{ij};\sigma)&= \pin{b}{y}+s\pin{b}{e_{ij}}
+\sigma\log\det(Q-\mathcal{A}^\top y -s\mathcal{A}^\top (e_{ij}))\\
&=\pin{b}{y}+s+\sigma\log\det(W^{-1} -sA_{ij}).
\end{align*}
Then, the gradient with respect to~$s$ is
\begin{equation}
\label{Eq:gradient}
\nabla_{s}f(y +se_{ij};\sigma)=1-\sigma\pin{A_{ij}}{(W^{-1} -sA_{ij})^{-1}}.
\end{equation}
Now the crucial task is to compute the inverse of the matrix
$W^{-1} -sA_{ij}$, which is of dimension~$n+1$. For this purpose,
notice that~$W^{-1}$ is changed by a rank-one or rank-two
matrix~$sA_{ij}$; see Lemma~\ref{lemma:rank}. Therefore, we can
compute both the inverse
matrix~$(W^{-1} -sA_{ij})^{-1}$ and the optimal step length
by means of the Sherman-Morrison formula 
for the rank-one or rank-two update; see Appendix~\ref{app1}.

Finally, we have to point out that the zero coordinate can also 
be chosen as ascent direction, in that case the gradient is
\[
\nabla_{s}f(y +se_0;\sigma)=1-\sigma\pin{A_0}{(W^{-1} -sA_0)^{-1}},
\]
and the computation of the step size is analogous.

\subsection{Algorithm overview}

Our approach to solve Problem~\eqref{P:Dsdp} is summarized in
Algorithm~\ref{algo:cd}.

\begin{algorithm}[H]
\SetAlgoRefName{CD}
\DontPrintSemicolon
\KwIn{$Q\in \Rbb^{(n+1)\times(n+1)}$}
\KwOut{A lower bound on the optimal value of Problem~\eqref{P:SDP-matrix}}
Use Lemma~\ref{lem:initial} to compute $y^{(0)}$ such that $Q-\mathcal{A}^\top y^{(0)} \succ 0$\;
Compute $W^{(0)} \gets (Q-\mathcal{A}^\top y^{(0)})^{-1}$\;
\For{$k \gets 0$ \textbf{until} max-iterations} {
Choose a coordinate direction $e_{ij^{(k)}}$ as described in Section~\ref{sec:direction}\;
Compute the step size $s^{(k)}$ as described in Section~\ref{sec:step}\label{cd:algo:step5}\;
Update $y^{(k+1)}\gets y^{(k)}+s^{(k)}e_{ij^{(k)}}$\label{cd:algo:step6}\;
Update $W^{(k)}$ using the Sherman-Morrison formula \;
Update $\sigma$\label{cd:algo:step8}\;
Terminate if some stopping criterion is met\label{cd:algo:step9}\;
}
\Return{$\pin{b}{y^{(k)}}$}\;
\caption{Barrier coordinate ascent algorithm for Problem~\eqref{P:Dsdp}}
\label{algo:cd}
\end{algorithm}

Before entering the main loop, the running time of
Algorithm~\ref{algo:cd} is dominated by the computation of the minimum
eigenvalue of~$\hat Q$ needed to compute~$y^{(0)}$ and by the
computation of the inverse matrix of~$Q-\mathcal{A}^\top
y^{(0)}$. Both can be done in~$O(n^3)$ time. Each iteration of the
algorithm can be performed in~$O(n^2)$. Indeed, as discussed in
Section~\ref{sec:direction}, we need to consider~$O(n)$ candidates for
the coordinate selection, so that this task can be performed
in~$O(n^2)$ time. For calculating the step size and updating the
matrix~$W^{(k)}$, we also need~$O(n^2)$ time using the
Sherman-Morrison formula.

Notice that the algorithm produces a feasible solution~$y^{(k)}$ of
Problem~\eqref{P:Dsdp} at every iteration and hence a valid lower
bound~$\pin{b}{y^{(k)}}$ for Problem~\eqref{P:SDP-matrix}. In
particular, when used within a branch-and-bound algorithm, this means
that Algorithm~\ref{algo:cd} can be stopped as soon
as~$\pin{b}{y^{(k)}}$ exceeds a known upper bound for
Problem~\eqref{P:SDP-matrix}.
Otherwise, the algorithm can be stopped after a fixed number of
iterations or when other criteria show that only a small further
improvement of the bound can be expected.

The choice of an appropriate termination rule however is closely
related to the update of~$\sigma$ performed in
Step~\ref{cd:algo:step8}. The aim is to find a good balance between
the convergence for fixed~$\sigma$ and the decrease of~$\sigma$. In
our implementation, we use the following rule: whenever the entry of the
gradient corresponding to the chosen coordinate has an absolute value
below~$0.01$, we multiply~$\sigma$
by~$0.25$. As soon as $\sigma$ falls below~$10^{-5}$, we fix it to
this value.

\subsection{Two-dimensional update}
\label{Sec:Simultaneus}

In Algorithm~\ref{algo:cd}, we change only one coordinate in each
iteration, as this allows to update the matrix~$W^{(k)}$ in $O(n^2)$
time using the Sherman-Morrison formula. This was due to the fact that
all constraint matrices in the primal SDP~\eqref{P:SDP-matrix} have
rank at most two. However, taking into account the special structure
of the constraint matrix~$A_0$, one can see that every linear
combination of any constraint matrix~$A_{ij}$ with~$A_0$ still has
rank at most two. In other words, we can simultaneously update the
dual variables~$y_0$ and~$y_{ij}$ and still recompute~$W^{(k)}$ in~$O(n^2)$ time.

In order to improve the convergence of Algorithm~\ref{algo:cd}, we
choose a coordinate~$ij$ as explained in Section~\ref{sec:direction}
and then perform an exact plane-search in the two-dimensional space
corresponding to the directions~$e_0$ and~$e_{ij}$, i.e., we
solve the bivariate problem
\begin{equation}
\label{P:simultLinesearch}
\argmax_{(s_0,s)}\,\{f(y+s_0e_0+se_{ij};\sigma)\mid
Q-\mathcal{A}^\top (y+s_0e_0+se_{ij}) \succ 0 , s\leq -y_{ij} \}\;,
\end{equation}
where we again omit the superscript~$(k)$ for sake of readibilty.
Similar to the one-dimensional case in~\eqref{P:Linesearch},
due to strict concavity of~$(s_0,s)\mapsto f(y+s_0e_0
+se_{ij};\sigma)$ over~$\{(s_0,s)\in\Rbb^2\mid Q-\mathcal{A}^\top
(y+s_0e_0+se_{ij}) \succ 0\}$, solving~\eqref{P:simultLinesearch} is
equivalent to finding the unique pair~$(s_0,s)\in\Rbb^2$ such that
$$\nabla_{s_0}f(y+s_0e_0+se_{ij};\sigma)=0$$
and either
\begin{equation*}
\nabla_{s}f(y+s_0e_0+se_{ij};\sigma)=0
\quad\text {and}\quad y_{ij}+s\leq 0
\end{equation*}
or 
\begin{equation*}
\nabla_{s}f(y+s_0e_0+se_{ij};\sigma)>0
\quad\text {and}\quad s=- y_{ij}.
\end{equation*}
To determine~$(s_0,s)$, it thus suffices to set both
gradients to zero and solve the resulting two-dimensional system of
equations. If it turns out that~$y_{ij}+s>0$, we
fix~$s:=- y_{ij}$ and recompute~$s_0$ by solving
\[
\nabla_{s_0}f(y+s_0e_0+se_{ij};\sigma) = 0.
\]
Proceeding as before, we have
\begin{align*}
f(y +s_0e_0 +se_{ij};\sigma)
&=\pin{b}{y}+s_0+ s+\sigma\log\det(W^{-1} -s_0A_0-sA_{ij}),
\end{align*}
and the gradients with respect to $s_0$ and $s$ are
\begin{align*}
\nabla_{s_0}f(y +s_0e_0+se_{ij};\sigma) = 1-\sigma\pin{A_0}{(W^{-1} -s_0A_0-sA_{ij})^{-1}}\\
\nabla_{s}f(y +s_0e_0+se_{ij};\sigma) = 1 -\sigma\pin{A_{ij}}{(W^{-1} -s_0A_0-sA_{ij})^{-1}}\;.
\end{align*}
The matrix $s_0A_0+sA_{ij}$ is of rank two; replacing~$(W^{-1}
-s_0A_0-sA_{ij})^{-1}$ by the Sherman-Morrison formula and setting the
gradients to zero, we obtain a system of two quadratic equations. For
details, see Appendix~\ref{app2}.
Using these ideas, a slightly different version of
Algorithm~\ref{algo:cd} is obtained by changing
Steps~\ref{cd:algo:step5} and~\ref{cd:algo:step6} adequately, which we
call Algorithm~CD2D.

\section{Experiments}
\label{Sec:Experiments}

For our experiments, we generate random instances in the same way as
proposed in~\cite{BuchheimWiegele(2013)}: the objective matrix is
$\hat Q = \sum_{i=1}^n\mu_iv_iv_i^\top$, where the $n$ numbers $\mu_i$
are chosen as follows: for a given value of $p\in[0,100]$, the first
$\nicefrac{pn}{100}$ $\mu_i$'s are generated uniformly from $[-1,0]$
and the remaining ones from $[0,1]$. Additionally, we generate $n$
vectors of dimension~$n$, with entries uniformly at random from
$[-1,1]$, and orthonormalize them to obtain the vectors $v_i$. The
parameter~$p$ represents the percentage of negative eigenvalues, so
that $\hat Q$ is positive semidefinite for $p=0$, negative
semidefinite for $p=100$ and indefinite for any other value~$p\in
(0,100)$. The entries of the vector~$\hat l$ are generated uniformly at
random from~$[-1,1]$, and~$\hat c=0$. In this paper, we restrict our
evaluation to ternary instances, i.e., instances with~$D_i=\{-1,0,1\}$.

We evaluate the performance of both Algorithms~\ref{algo:cd} and~CD2D
in the root node of the branch-and-bound tree and compare them with
CSDP, the SDP solver used in~\cite{BuchheimWiegele(2013)}. Our
experiments were performed on an Intel Xeon processor running at
2.5\;GHz. Algorithms~\ref{algo:cd} and~CD2D were implemented in C++,
using routines from the LAPACK package only in the initial phase for computing a
starting point and the inverse matrix~$W^{(0)}$.

The main motivation to consider a fast coordinate ascent method was to
obtain quick and good lower bounds for the quadratic integer
problem~\eqref{P:Quadratic}. We are thus interested in the improvement
of the lower bound over time. In Figure~\ref{Fig:TernatyLB}, we
plotted the lower bounds obtained by CSDP and by the
algorithms~\ref{algo:cd} and~CD2D in the root node for two ternary
instances of size~$n=100$, for the two values $p=0$
and~$p=100$. Notice that we use a log scale for the $y$-axis.
\begin{figure}[h!]
\centering
\includegraphics[scale=0.65]{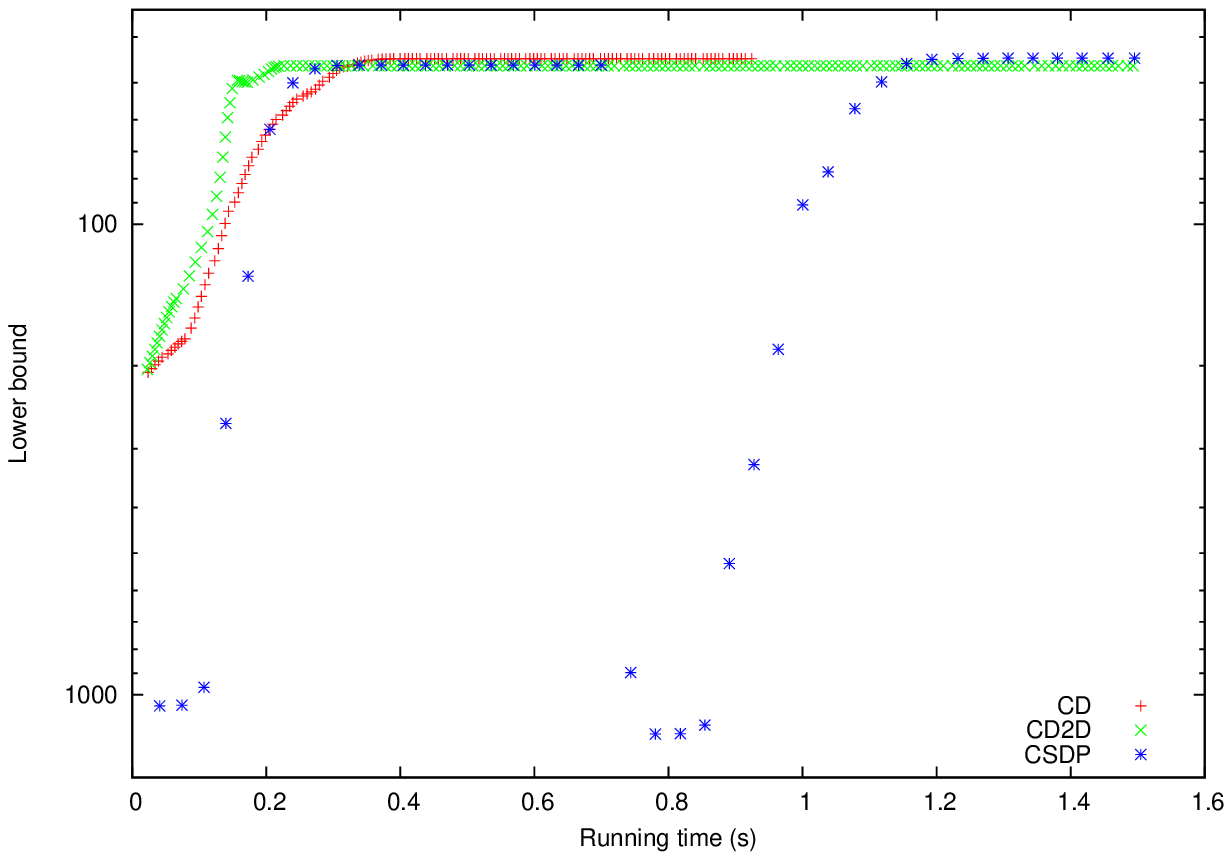}
\includegraphics[scale=0.65]{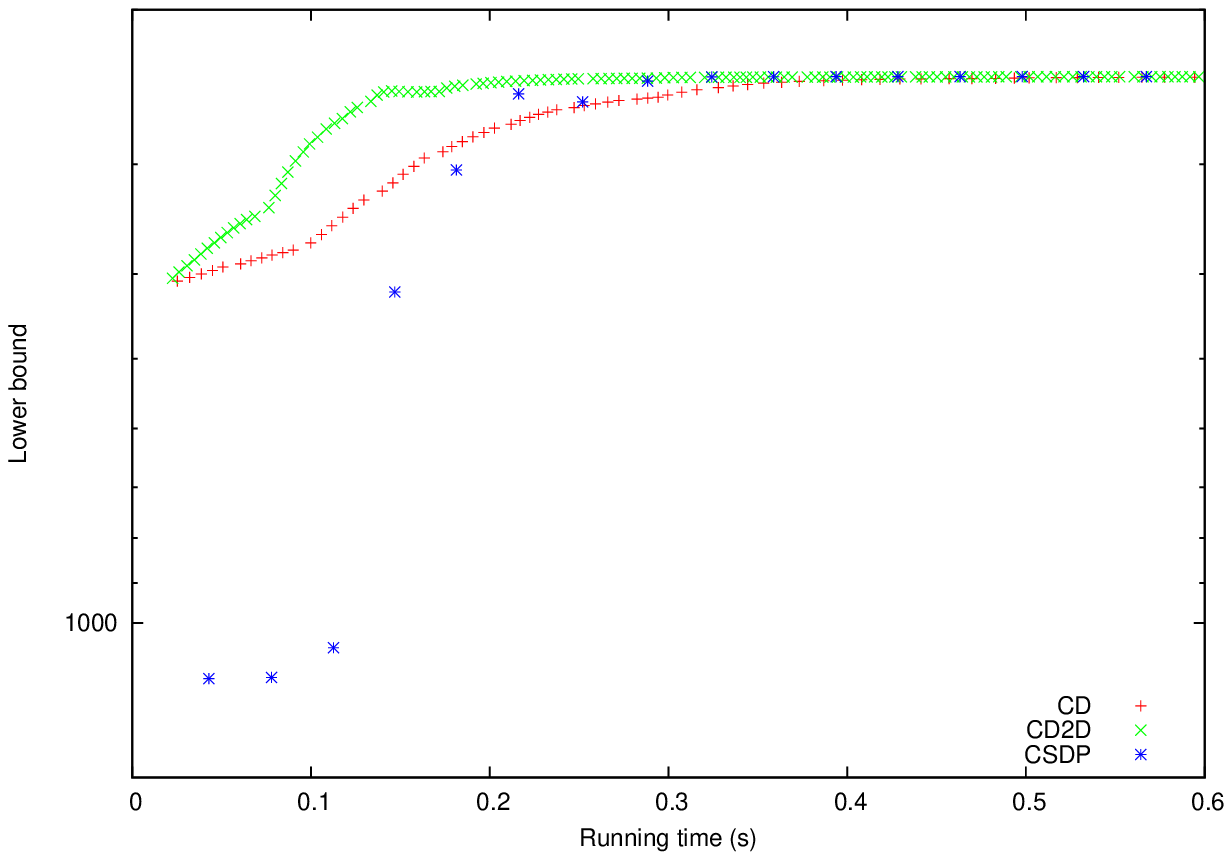}
\caption{\label{Fig:TernatyLB} \small Comparison of the lower bounds
  in the root node obtained by Q-MIST with CSDP, CD and CD2D; for
  $p=0$ (top) and $p=100$ (bottom)}
\end{figure}

From Figure~\ref{Fig:TernatyLB}, we see that Algorithm~CD2D
clearly dominates both other approaches: the lower bound it produces exceeds the
other bounds until all approaches come close to the optimum
of~\eqref{P:SDP-relax}. This is true in particular for the instance
with~$p=100$. Even Algorithm~\ref{algo:cd} is stronger than CSDP in the
beginning, but then CSDP takes over.
Note that the computation of the root bound for the instance shown in
Figure~\ref{Fig:TernatyLB}\;(a) involves one re-optimization due to
separation. For this reason, the lower bound given by CSDP has to
restart with a very weak value.

As a next step, we will integrate the Algorithm~CD2D into the
branch-and-bound framework of Q-MIST. We are confident that this will
improve the running times of Q-MIST significantly when choosing the
stopping criteria carefully. This is left as future work.

\bibliographystyle{plain}
\bibliography{Base1}

\newpage

\appendix

\section{Feasible starting point}
\label{app:feas}

\begin{proof} (of Lemma~\ref{lem:initial})
  We have $y^{(0)}_{ij}\leq0$ by construction, so it remains to
  show that~$Q-\mathcal{A}^\top y^{(0)}\succ0$. To this end, first note that
\begin{equation}\label{ctilde}
\tilde c:=\hat c-y_{0}-\tilde y
\sum_{i=1}^n(1+l_iu_i)=1+(\tfrac{1}{2}\hat l-\tilde{y}a)^{\top}(\tfrac{1}{2}\hat l-\tilde{y}a)>
0\;.
\end{equation}
By definition,
\begin{align*}
\label{Eq:001}
Q-\mathcal{A}^\top y^{(0)} &=Q-y_0A_0-\tilde y\sum_{i=1}^nA_{iu_i} 
\\ &= Q-y_0A_0-\tilde y
\begin{pmatrix}
\sum_{i=1}^n(1+l_iu_i) & a^\top\\
a & I_n
\end{pmatrix}
\\ &=
\begin{pmatrix}
\tilde c & (\tfrac{1}{2}\hat l-\tilde{y}a)^\top\\
\tfrac{1}{2}\hat l-\tilde{y}a & \hat Q-\tilde yI_{n}
\end{pmatrix},
\end{align*}
which by Schur complement and~\eqref{ctilde} is positive definite if
\[
(\hat Q -\tilde y I)-\tfrac{1}{\tilde c}(\tfrac{1}{2}\hat l-\tilde{y}a)(\tfrac{1}{2}\hat l-\tilde{y}a)^{\top}\succ 0.
\]
Denoting~$B:=(\tfrac{1}{2}\hat l-\tilde{y}a)(\tfrac{1}{2}\hat l-\tilde{y}a)^{\top}$,
we have
$$\lambda_{max}(B)=(\tfrac{1}{2}\hat l-\tilde{y}a)^{\top}(\tfrac{1}{2}\hat l-\tilde{y}a)\geq0$$
and thus
\begin{align*}
\lambda_{min}\left((\hat Q -\tilde y I_n)-\tfrac{1}{\tilde c}B\right)
&\geq\lambda_{min}(\hat Q -\tilde y I_n)+\tfrac{1}{\tilde c}\lambda_{min}(-B)\\
&=\lambda_{min}(\hat Q) -\tilde
y-\frac{\lambda_{max}(B)}{1+\lambda_{max}(B)}>0
\end{align*}
by definition of~$\tilde y$.
\qed
\end{proof}

\section{Computation of the step size}

\subsection{One-dimensional problem}\label{app1}

We need to find the value of~$s$ such that the gradient in 
\eqref{Eq:gradient} is zero. For this we need to solve the following equation:
\begin{equation}
\label{Eq:gradientzero}
1-\sigma\pin{A_{ij}}{(W^{-1} -sA_{ij})^{-1}}=0.
\end{equation}
Notice that each constraint matrix $A_{ij}$ can be factored as follows:
\[
A_{ij}=E_{ij}IC_{ij},
\]
where $E_{ij}\in \Rbb^{(n+1)\times2}$, defined by $E_{ij}:=(e_{0}\; e_{i})$, $e_{0},
e_{i}\in \Rbb^{n+1}$, $C\in \Rbb^{2\times(n+1)}$ defined by $C:=(A_{ij})_{\{0,i\},\{0,\dots,n\}}$ and $I$ is the $2\times2$-identity matrix. 
As mentioned in Section~\ref{sec:step}, the inverse 
matrix~$(W^{-1} -sA_{ij})^{-1}$ can be computed using the Sherman-Morrison
formula
as follows: 
\[
(W^{-1} -sA_{ij})^{-1}=(W^{-1} -sE_{ij}IC_{ij})^{-1}
=W+WE_{ij}(\tfrac{1}{s}I+C_{ij}WE_{ij})^{-1}C_{ij}W\;.
\]
Notice that the matrix $\tfrac{1}{s}I+C_{ij}WE_{ij}$ is a $2\times2$-matrix,
so its inverse can be easily computed. 
Replacing the inverse in~\eqref{Eq:gradientzero}, we get
\[
1-\sigma\pin{A_{ij}}{W}-\sigma\pin{A_{ij}}{WE_{ij}(\tfrac{1}{s}I+C_{ij}WE_{ij})^{-1}C_{ij}W}=0.
\]
Due to the sparsity of the constraint matrices $A_{ij}$, the inner matrix product 
is simplified a lot, in fact we have to compute only the entries $00$, $0i$,
$0i$ and $ii$ of the  matrix product
$WE_{ij}(\tfrac{1}{s}I+C_{ij}WE_{ij})^{-1}C_{ij}W$. We arrive at a quadratic 
equation in $s$, namely 
\[
as^2+bs+c=0,
\]
where 
\begin{eqnarray*}	 
a & = & -(A_{ij})_{0i}^2w_{0i}^2 + (A_{ij})_{00}(A_{ij})_{ii}w_{0i}^2 
+(A_{ij})_{0i}(A_{ij})_{0i}w_{00}w_{ii}\\
&&-(A_{ij})_{00}(A_{ij})_{ii}w_{00}w_{ii},\\
b & = & (A_{ij})_{00} w_{00} +2 (A_{ij})_{0i} w_{0i} 
- 2 \sigma (A_{ij})_{0i}^2 w_{0i}^2 + 2 \sigma (A_{ij})_{00} (A_{ij})_{ii}  w_{0i}^2 \\
&&+ (A_{ij})_{ii} w_{ii} +2\sigma (A_{ij})_{0i}^2w_{00} w_{ii} -2\sigma (A_{ij})_{00} (A_{ij})_{ii} w_{00} w_{ii},\\
c &=& -1 + \sigma (A_{ij})_{00} w_{00} +2 \sigma (A_{ij})_{0i} w_{0i}+  \sigma (A_{ij})_{ii} w_{ii}.
\end{eqnarray*}
Finally, $s$ is obtained using the well-known formula for the roots of a
general quadratic equation.

The computation of the step size becomes simpler if the chosen
coordinate direction corresponds to~$y_0$. We then need to find a
solution of the equation
\begin{equation}
\label{eq:shermanmorrisony0}
1-\sigma\pin{A_0}{(W^{-1} -sA_0)^{-1}}=0.
\end{equation}
The inverse of $W^{-1} -sA_0$ is represented using the
Sherman-Morrison formula for rank-one,
\begin{align*}
(W^{-1} -sA_{0})^{-1} =(W^{-1} - se_{0}e_{0}^\top)^{-1} 
= W-\frac{s}{1+sw_{ii}} (We_{i}) (We_{i})^\top.
\end{align*}
Using this to solve \eqref{eq:shermanmorrisony0}, we obtain the step size
\[
s=\frac{1}{w_{ii}}-\sigma.
\]
A similar formula for the step size is obtained for other cases when
the constraint matrix $A_{ij}$ has rank one.

\subsection{Two-dimensional problem}\label{app2}

We write $s_0A_0+sA_{ij}= E_{ij}IC_{ij}$, where $E_{ij} = (e_0 \;e_i)\in
\Rbb^{(n+1)\times 2}$, and
\[
C_{ij} =
\begin{pmatrix}
s_0+s(A_{ij})_{00} & \dots & s(A_{ij})_{0i} & \dots \\
s(A_{ij})_{0i} & \dots & s(A_{ij})_{ii} & \dots 
\end{pmatrix}
\in\Rbb^{2\times(n+1)}.
\]
To compute the inverse matrix $(W^{-1} -s_0A_0-sA_{ij})^{-1}$ we use the
Sherman-Morrison formula again, obtaining
\[
(W^{-1} -s_0A_0-sA_{ij})^{-1} = 
(W^{-1} -E_{ij}IC_{ij})^{-1} = W+WE_{ij}(I+C_{ij}WE_{ij})^{-1}C_{ij}W.
\]
Substituting this in the gradients and setting them to zero, we obtain the 
following system of two quadratic equations
\begin{eqnarray*}
\sigma\pin{A_0}{(W^{-1} -s_0A_0-sA_{ij})^{-1}} &=& 1\\
\sigma\pin{A_{ij}}{(W^{-1} -s_0A_0-sA_{ij})^{-1}} &=& 1,
\end{eqnarray*}
the solutions of which are
$(s_0',s')$ and $(s_0'',s'')$ given as follows:
\begin{align*}
\begin{split}
s_0'& = -(-4 (A_{ij})_{0i}^3 (A_{ij})_{ii} w_{0i} w - 4 \alpha (A_{ij})_{0i} (A_{ij})_{ii}^2 w_{0i} w -  4 (A_{ij})_{0i}^4 w_{00} w \\
&\phantom{=~}+ 2 (A_{ij})_{0i}^2 (3 (A_{ij})_{00} (A_{ij})_{ii} w_{00} w - 2 (A_{ij})_{ii}  w_{00} w - (A_{ij})_{ii}^2 w (w_{ii} + \sigma w) \\
&\phantom{=~} +\rho) + (A_{ij})_{ii} (-2 (A_{ij})_{00} (A_{ij})_{00} (A_{ij})_{ii} w_{00} w + 2 (A_{ij})_{00} ((A_{ij})_{ii}  w_{00} w\\
&\phantom{=~}-(A_{ij})_{ii}^2 w (w_{ii}+\sigma w) + \rho) +  (-(A_{ij})_{ii}^2 w (2 w_{ii} + \sigma w) +\rho)))/\delta,\\
s' & = (2 (A_{ij})_{0i}^2 w_{00} w_{0i}^2- 2 (A_{ij})_{00} (A_{ij})_{ii} w_{00} w_{0i}^2+ 2 (A_{ij})_{ii}  w_{00} w_{0i}^2+\sigma (A_{ij})_{ii}^2 w_{0i}^4\\
&\phantom{=~}-2 (A_{ij})_{0i}^2 w_{00}^2 w_{ii} + 2 (A_{ij})_{00} (A_{ij})_{ii} w_{00}^2 w_{ii} -2 (A_{ij})_{ii}  w_{00}^2 w_{ii} \\
&\phantom{=~}- 2\sigma (A_{ij})_{ii}^2  w_{00} w_{0i}^2 w_{ii} \sigma (A_{ij})_{ii}^2 w_{00}^2 w_{ii}^2-\rho)/\delta,\\
s_0'' & = (4 (A_{ij})_{0i}^4 w_{00} w + 4 (A_{ij})_{0i}^3 (A_{ij})_{ii} w_{0i} w +4\alpha (A_{ij})_{0i} (A_{ij})_{ii}^2 w_{0i} w \\
&\phantom{=~}+2 (A_{ij})_{0i}^2 (2 (A_{ij})_{ii}  w_{00} w- 3 (A_{ij})_{00} (A_{ij})_{ii} w_{00}w + (A_{ij})_{ii}^2 w (w_{ii} + \sigma w)\\
&\phantom{=~} +\rho) - (A_{ij})_{ii} (-2 (A_{ij})_{00}^2 (A_{ij})_{ii} w_{00} w+ 2 (A_{ij})_{00} ((A_{ij})_{ii}  w_{00} w\\
&\phantom{=~} +(A_{ij})_{ii}^2 w (w_{ii} + \sigma w) + \rho) -  ((A_{ij})_{ii}^2 w (2 w_{ii}+ \sigma w) + \rho)))/\delta,\\
s'' & = (2 (A_{ij})_{0i}^2 w_{00} w_{0i}^2- 2 (A_{ij})_{00} (A_{ij})_{ii} w_{00} w_{0i}^2+ 2 (A_{ij})_{ii}  w_{00} w_{0i}^2+ \sigma (A_{ij})_{ii}^2 w_{0i}^4\\
&\phantom{=~}-2 (A_{ij})_{0i}^2 w_{00}^2 w_{ii} + 2 (A_{ij})_{00} (A_{ij})_{ii} w_{00}^2 w_{ii} -2 (A_{ij})_{ii}  w_{00}^2 w_{ii} \\
&\phantom{=~}- 2 \sigma (A_{ij})_{ii}^2 w_{00} w_{0i}^2 w_{ii} + \sigma (A_{ij})_{ii}^2 w_{00}^2 w_{ii}^2+\rho)/\delta.
\end{split}
\end{align*}
Here we set
\begin{align*}
w &= w_{0i}^2 -w_{00}w_{ii},\\
\alpha &= -(A_{ij})_{00}+1,\\
\begin{split}
\rho^2 &= w^2 (4 (A_{ij})_{0i}^4 w_{00}^2 + 8 (A_{ij})_{0i}^3 (A_{ij})_{ii} w_{00} w_{0i} + 8\alpha(A_{ij})_{0i} (A_{ij})_{ii}^2  w_{00} w_{0i} \\
&\phantom{=~}+ 4 (A_{ij})_{0i}^2 (A_{ij})_{ii} (\alpha w_{00}^2+ (A_{ij})_{ii} w_{0i}^2) +(A_{ij})_{ii}^3 (-4 (A_{ij})_{00} w_{0i}^2 +4  w_{0i}^2\\
&\phantom{=~}+ \sigma^2(A_{ij})_{ii}  w^2)),
\end{split}
\\
\delta &= -2 (A_{ij})_{ii}^2 ((A_{ij})_{0i}^2 + \alpha(A_{ij})_{ii}) w^2.
\end{align*}

\end{document}